\documentclass[final,reqno,dvipdfmx]{amsart}

\usepackage{fullpage}
\usepackage[foot]{amsaddr}
\usepackage{amsmath,mathtools,amssymb}
\usepackage{mathrsfs,euscript}
\usepackage[only llbracket,rrbracket]{stmaryrd} 
\usepackage{tikz}
\usetikzlibrary{cd}
\usepackage{showkeys}
\usepackage[bookmarks=false,draft=false,breaklinks,colorlinks]{hyperref}
\usepackage{cleveref}
\usepackage{amsthm}
\usepackage{autonum}

\numberwithin{equation}{section}
\allowdisplaybreaks

\newenvironment{Ack}%
{\par \vspace{\baselineskip}%
 \noindent \textbf{Acknowledgements.}}%
{\par \vspace{\baselineskip}}

\usepackage{enumitem}
\setenumerate{label=(\arabic*),nosep}
\setitemize{nosep}
\newlist{clist}{enumerate}{1}
\setlist*[clist]{label=(\roman*), nosep}

\crefname{section}{\S\!}{\S\S\!}
\crefname{subsection}{\S\!}{\S\S\!}
\crefname{thm}{Theorem}{Theorems}
\crefname{dfn}{Definition}{Definitions}
\crefname{prp}{Proposition}{Propositions}
\crefname{lem}{Lemma}{Lemmas}
\crefname{cor}{Corollary}{Corollaries}
\crefname{rmk}{Remark}{Remarks}
\crefname{eg}{Example}{Examples}

\theoremstyle{definition}
\newtheorem{thm}{Theorem}[section]
\newtheorem{dfn}[thm]{Definition}
\newtheorem{prp}[thm]{Proposition}
\newtheorem{lem}[thm]{Lemma}

\newtheorem{rmk}[thm]{Remark}

\newtheorem*{mth}{Theorem}
\newtheorem*{dfn*}{Definition}
\newtheorem*{rmk*}{Remark}

\newcommand{\pd}{\partial}
\newcommand{\bl}{\bullet}
\newcommand{\ve}{\varepsilon}

\newcommand{\ol}{\overline}

\newcommand{\wt}{\widetilde}
\newcommand{\ceq}{\coloneqq} 

\newcommand{\xrr}[1]{\xrightarrow{\ #1 \ }{}}

\newcommand{\lto}{\longrightarrow}

\newcommand{\mto}{\mapsto}
\newcommand{\lmto}{\longmapsto}

\newcommand{\bfk}{\mathbf{k}}
\newcommand{\bfx}{\mathbf{x}}

\newcommand{\bbN}{\mathbb{N}}
\newcommand{\bbQ}{\mathbb{Q}}

\newcommand{\bbZ}{\mathbb{Z}}
\newcommand{\clC}{\mathcal{C}}
\newcommand{\clG}{\mathcal{G}}

\newcommand{\frS}{\mathfrak{S}}
\newcommand{\con}{\textup{con}}

\newcommand{\sL}{\mathbf{\Lambda}}
\newcommand{\sN}{\mathrm{sNSym}}
\newcommand{\sQ}{\mathrm{sQSym}}
\newcommand{\SVec}{\mathsf{SVec}}
\newcommand{\catC}{\mathsf{C}}

\newcommand{\ev}{\ol{0}}
\newcommand{\od}{\ol{1}}
\newcommand{\Zt}{\bbZ/2\bbZ}
\newcommand{\ch}{\textup{ch}}

\newcommand{\tboplus}{{\textstyle\bigoplus}}

\newcommand{\abs}[1]{\left| #1 \right|}
\newcommand{\rst}[2]{\left. #1 \right|_{#2}}

\DeclareMathOperator{\id}{id}
\DeclareMathOperator{\Aut}{Aut}
\DeclareMathOperator{\Hom}{Hom}

\DeclareMathOperator{\Spec}{Spec}

\begin{document}

\title{On the Hopf superalgebra of symmetric functions in superspace}
\author{Masamune Hattori, Renta Yagi, Shintarou Yanagida}
\date{March 3, 2025. Revised: July 24, 2025.}
\address{Graduate School of Mathematics, Nagoya University. Furocho, Chikusaku, Nagoya, Japan, 464-8602.}
\email{m21039e@math.nagoya-u.ac.jp, renta.yagi.e6@math.nagoya-u.ac.jp, yanagida@math.nagoya-u.ac.jp} 

\begin{abstract}
We introduce a superspace analogue of combinatorial Hopf algebras (Aguiar--Bergeron--Sottile, 2006),
and show that the Hopf superalgebra of quasi-symmetric (resp.\ symmetric) functions in superspace 
(Fishel--Lapointe--Pinto, 2019)  is a terminal object in the category of 
all (resp.\ cocommutative) combinatorial Hopf superalgebras. 
We also introduce a superspace analogue of chromatic symmetric functions of graphs (Stanley, 1995)
using the chromatic Hopf superalgebra of two-colored graphs.
\end{abstract}

\maketitle
{\small \tableofcontents}

\section{Introduction}\label{s:intro}

\subsection*{Background}

The ring of \emph{symmetric functions} \cite[Chap.\ I]{M}, \cite[Chap.\ 7]{Sb} 
is a commonplace in several branches of mathematics, including representation theory, 
(quantum) integrable systems, and combinatorics.
It consists of the formal power series in the variables $x=(x_1,x_2,\dotsc)$,
and the ring structure is given by addition and multiplication of such series.
A closely related object is the ring of \emph{quasi-symmetric functions} \cite{G}, 
which can be seen as a refinement of symmetric functions and provides many applications 
in enumeration problems in combinatorics and algebraic geometry.

Both of these rings have the \emph{Hopf algebra} structures \cite[Chap.\ 2, Chap.\ 5]{GR}. 
Among the large number of studies on these Hopf algebras, 
let us recall the work of Aguiar, Bergeron and Sottile \cite{ABS}.
They introduced the notion of a \emph{combinatorial Hopf algebra} 
as a graded connected Hopf algebra $H$ over a field $\bfk$ 
equipped with an algebra morphism $\zeta\colon H \to \bfk$ called 
a \emph{character}, and showed that the Hopf algebra of quasi-symmetric (resp.\ symmetric) functions
is a terminal object in the category of all (resp.\ cocommutative) combinatorial Hopf algebras.
It gives a characterization of the Hopf algebras of (quasi-)symmetric functions, and 
explains the ubiquity of (quasi-)symmetric functions in representation theory and combinatorics.
See also \cite[Chap.\ 7]{GR} for further explanation of the Aguiar--Bergeron--Sottile theory.

Over the last 20 years, a \emph{superspace extension} of the theory of symmetric functions 
has been developed and has received much attention \cite{A+,B+,DLM1,DLM,DLM3}.
In the superspace setting, one considers formal power series in \emph{supervariables} 
$(x;\theta)=(x_1,x_2,\dotsc;\theta_1,\theta_2,\dotsc)$, 
where the $x_i$ are ordinary commuting variables, 
while the $\theta_i$ are the anti-commuting variables satisfying 
$\theta_i \theta_j=-\theta_j\theta_i$ and $\theta_i^2=0$.
Considering the ``diagonal permutation of supervariables'' 
(see \cref{ss:sLsQ:sL}, \eqref{eq:sL:SN-act} for the precise definition), 
one has the notion of symmetric functions 
in superspace \cite[\S2.2, \S2.3]{DLM}, \cite[\S2]{A+}.

Let us also mention the motivation from physics.
Symmetric functions are important not only in pure mathematics, 
but also in connection with mathematical physics. 
For example, Schur and Jack symmetric functions appear in
the Calogero–-Moser-–Sutherland (quantum) integrable systems \cite[\S2]{DLM},
while Macdonald symmetric functions appear in the Macdonald--Ruijsenaars
quantum integrable systems \cite[Chap.\ VI]{M}, \cite{N}.
The papers \cite{B+,DLM1,DLM,DLM3} mentioned above are motivated by 
a supersymmetric generalization of this connection.

Recently, Fishel, Lapointe and Pinto \cite{FLP} investigated a superspace analogue 
of the rich connection between symmetric function theory and quasi-symmetric functions,
and revealed the Hopf algebra structure of the symmetric functions in superspace $\sL$
and of the quasi-symmetric functions in superspace $\sQ$.
The Hopf dual of $\sQ$, denoted by $\sN$ in \cref{s:CHSA}, 
is further studied in the recent paper \cite{AGM}.
These Hopf superalgebras (see \cref{dfn:HSA} for the precise meaning) are natural 
super-analogues of the Hopf algebras of (quasi-)symmetric functions.
We will review these topics in \cref{s:sLsQ}.

\subsection*{Main results}

In this paper, we investigate a superspace analogue of Aguiar--Bergeron--Sottile theory,
and give a characterization of the Hopf superalgebras of (quasi-)symmetric functions in superspace.

Let us explain the idea of our superspace analogue.
Since a combinatorial Hopf algebra $H$ is equipped with a character $\zeta\colon H \to \bfk$,
we need a natural superspace analogue of a character.
Now, we recall some basic theory of (affine) super algebraic geometry \cite[\S1, \S2]{KV}.
A character $\zeta$ can be regarded as a morphism 
from the point $\Spec \bfk$ to ``the non-commutative scheme $\Spec H$''.
According to \cite[\S2.2]{KV}, as a superspace analogue of the point $\Spec \bfk$, 
we can consider the ``$\mathcal{N}=1$ supersymmetric particle'' $\Spec \bfk[\ve]$, 
where $\bfk[\ve]=\bfk+\bfk\ve$ is the exterior algebra of one variable $\ve$ 
(the commutative superalgebra in one odd variable $\ve$). 
This leads us to the following definition.

\begin{dfn*}[\cref{dfn:CHSA}]
Let $H$ be a Hopf superalgebra over $\bfk$.
\begin{enumerate}
\item 
An even superalgebra morphism $\zeta\colon H \to \bfk[\ve]$  
is called a \emph{supercharacter} of $H$.  

\item 
$H$ is called a \emph{combinatorial Hopf superalgebra} 
if it is connected graded (see \cref{dfn:HSA:gr})
and has a supercharacter $\zeta\colon H \to \bfk[\ve]$.
\end{enumerate}
\end{dfn*}

We will see in \cref{s:CHSA} that $\sL$ and $\sQ$ are combinatorial Hopf superalgebras 
with supercharacter $\zeta_S$ \eqref{eq:zeta_S} and $\zeta_Q$ \eqref{eq:zeta_Q}. 
The main results of this note are:

\begin{mth}[\cref{thm:CHSA:sQ}, \cref{prp:CHSA:sL}]\label{thm:main}
The Hopf superalgebra $\sQ$ (resp.\ $\sL$) is a terminal object 
in the category of all (resp.\ cocommutative) combinatorial Hopf superalgebras.
In other words, we have:

Let $A=\bigoplus_{k\ge0}A^k$ be a combinatorial Hopf superalgebra 
with supercharacter $\zeta\colon A \to \bfk[\ve]$.
\begin{enumerate}
\item 
There is a unique graded even morphism $\Psi\colon A \to \sQ$ of graded Hopf superalgebras 
such that $\zeta=\zeta_Q\circ\Psi$, where $\zeta_Q$ is the supercharacter \eqref{eq:zeta_Q}.

\item \label{i:main:2}
If $A$ is cocommutative, then there is a unique graded even morphism 
$\Psi\colon A \to \sL$ of graded Hopf superalgebras 
such that $\zeta=\zeta_S\circ\Psi$, where $\zeta_S$ is the supercharacter \eqref{eq:zeta_S}.
\end{enumerate}
\end{mth}

The proof is based on techniques similar to those used in the non-super case \cite{GR}, 
with the key difference being the use of the Hopf superalgebra $\sN$ of 
noncommutative symmetric functions in superspace (see \cite[\S6]{FLP} and \cite{AGM} 
for details) instead of the Hopf algebra $\operatorname{NSym}$ of 
noncommutative symmetric functions used in the non-super case.

As another example of a combinatorial Hopf superalgebra, 
in the final \cref{s:chr} we will consider a superspace analogue of 
chromatic symmetric functions of graphs \cite{S}.
Extending the non-superspace case \cite[\S7.3]{GR}, we consider the space $\clG$ 
spanned by the basis elements $[G]$ associated to two-colored graphs $G=(V,W,E)$ 
(the set $W$ of white vertices, the set $V \setminus W$ of black vertices
 and the set $E$ of edges),
which is a cocommutative combinatorial Hopf superalgebra (\cref{thm:chr:clG}).
Then, by the \cref{thm:main} \ref{i:main:2} above, we have a graded even morphism
$\Psi_{\ch}\colon \clG \to \sL$ of graded Hopf superalgebras.

\begin{dfn*}[\cref{dfn:chr:PsiG}]
For a two-colored graph $G$, we call $\Psi_{\ch}([G])$ 
the \emph{chromatic symmetric function in superspace} associated to $G$.
\end{dfn*}

If $G=(V,W,E)$ satisfies $W=\emptyset$, 
then $\Psi_{\ch}([G])$ is equal to Stanley's chromatic symmetric function of the graph $(V,E)$.
We will show in \cref{prp:chr:PsiG} 
an expansion formula of chromatic symmetric functions in superspace.

To conclude this introduction, we would like to discuss the naturalness 
of our formulation of combinatorial Hopf superalgebras.
First, as explained above, it naturally extends the combinatorial Hopf algebras of 
symmetric functions and chromatic symmetric functions.
Second, our definition of a supercharacter is a natural extension of the notion of 
a character for a (non-super) combinatorial Hopf algebra,  
viewed from the perspective of supergeometry.
Third, as will be discussed below, our notion of supercharacter has the potential 
to contribute to the combinatorial study of symmetric functions on superspaces.
In the non-super case, quasi-symetric functions are useful for studying certain 
combinatorial objects related to symmetric functions.
As a chromatic analogue, Shareshian and Wachs \cite{SW} introduced 
chromatic quasi-symmetric functions, which have several applications 
in the combinatorial study of chromatic symmetric functions.
For example, Kaliszewski \cite{K} derived the hook coefficient formula for 
chromatic symmetric functions by considering chromatic quasi-symmetric functions. 
In the super case, we believe that 
a super analogue of quasi-symmetric functions can be introduced.
Our supercharacter may serve as a key concept for computing combinatorial objects in superspace, 
such as hook coefficients of chromatic symmetric functions in superspace.

\subsection*{Organization}

Let us explain the organization of this note.

\cref{s:HSA} and \cref{s:sLsQ} are preliminary.
In \cref{s:HSA}, we fix our terminology on Hopf algebras in super setting.
In \cref{s:sLsQ}, we review the Hopf superalgebras $\sL$ and $\sQ$ of 
symmetric and quasi-symmetric functions in superspace, 
largely following \cite{DLM}, \cite{A+} and \cite{FLP}.

\cref{s:CHSA} is the main part of this paper.
We introduce the notion of a combinatorial Hopf superalgebra in \cref{dfn:CHSA}, 
and characterize $\sQ$ (resp.\ $\sL$) as a terminal object of the category of 
all (resp.\ cocommutative) Hopf superalgebras in \cref{thm:CHSA:sQ} (resp.\ \cref{prp:CHSA:sL}).
This can be seen as a superspace extension of the Aguiar--Bergeron--Sottile theory.

The final \cref{s:chr} gives a superspace analogue of chromatic symmetric functions.
We introduce the chromatic Hopf superalgebra 
and show that it is combinatorial (\cref{thm:chr:clG}).
Then, in \cref{dfn:chr:PsiG}, 
we introduce the notion of a chromatic symmetric function in superspace
using the universality of $\sL$, and study some examples (\cref{prp:chr:Kn1}) 
and the expansion formula (\cref{prp:chr:PsiG}).

\subsection*{Global notation}

The following list gives an overview of the terminology and notations used in the text.
\begin{itemize}
\item
The symbol $\bbN$ denotes the set $\{0,1,2,\dotsc\}$ of all nonnegative integers.

\item
We use the standard symbols $\bbZ$ and $\bbQ$ of integers and rational numbers.

\item
For a (finite) set $S$, we denote its cardinality by $\# S$. 

\item 
A ring or an algebra means a unital associative one unless otherwise stated.

\item
A coalgebra means a counital coassociative one unless otherwise stated.

\item
The symbol $\pd_x$ denotes the partial differential $\frac{\pd}{\pd x}$ with respect to $x$.

\item
The symbol $\delta_{i,j}$ denotes the Kronecker delta.

\item
We follow \cite{DM} and \cite[\S1]{BE} for the terminology of super mathematics.
We denote the cyclic group of order $2$ by $\Zt=\{\ev,\od\}$,
and a $\Zt$-grading is called \emph{parity}.
The parity of an object $x$ is denoted by $\abs{x} \in \Zt$.
\end{itemize}

\section{Hopf superalgebra}\label{s:HSA}

Let us fix the terminology and notation of what we will call Hopf superalgebras.

Let $\bfk$ be a field of characteristic $0$,
and consider linear spaces and linear maps over $\bfk$.
A superspace means a $\Zt$-graded linear space.
For a superspace $V$, 
the even and odd homogeneous parts are denoted by $V_{\ev}$ and $V_{\od}$, respectively.
A linear map $f\colon V \to W$ between superspaces $V$ and $W$ is even (resp.\ odd)
if it preserves (resp.\ reverses) the parity of elements. 
A linear map $f$ decomposes as 
$f=f_{\ev}+f_{\od}$ with $f_{\ev}$ even and $f_{\od}$ odd.
This makes the set $\Hom(V,W)$ of linear maps into a superspace.
Then, superspaces and linear maps between them form a supercategory $\SVec$ 
in the sense of \cite[Definition 1.1]{BE}.

$\SVec$ has the standard structure of a monoidal supercategory \cite[Definition 1.4]{BE}
with tensor product $(V \otimes W)_{\ev}\ceq V_{\ev} \otimes W_{\ev}+V_{\od}\otimes W_{\od}$
and $(V \otimes W)_{\od}\ceq V_{\ev} \otimes W_{\od}+V_{\od}\otimes W_{\ev}$ 
for superspaces $V,W$, 
and $(f \otimes g)(v \otimes w)\ceq (-1)^{\abs{g}\abs{v}}f(v) \otimes g(w)$
for linear maps $f,g$ and $v\in V$, $w \in W$.
Here and hereafter we denote by $\abs{x} \in \Zt$ the parity of the object $x$.
Note that the composition of tensor products of linear maps has a sign:
$(f\otimes g)\circ(h \otimes i)=(-1)^{\abs{g}\abs{h}}(f \circ h)\otimes (g \circ i)$.

Given a monoidal supercategory $\catC$, 
one has the notions of a \emph{superalgebra object},
of a (\emph{super}-)\emph{coalgebra object},
and of a (\emph{super}-)\emph{bialgebra object} in $\catC$.
Now, let us introduce: 

\begin{dfn}\label{dfn:HSA}
A \emph{Hopf superalgebra} $H$ (over $\bfk$) is a Hopf algebra object 
in the monoidal supercategory $\SVec$.
\end{dfn}

The monoidal supercategory $\SVec$ is symmetric 
with braiding $u \otimes v \mto (-1)^{\abs{u}\abs{v}}v \otimes u$.
Hence, we have the notions of a \emph{commutative} Hopf superalgebra and 
of a \emph{cocommutative} Hopf superalgebra.

In preparation for \cref{s:CHSA}, let us also introduce:

\begin{dfn}\label{dfn:HSA:gr}
Let $H$ be a Hopf superalgebra.
\begin{enumerate}
\item 
$H$ is called \emph{graded} if it has an additional $\bbN$-grading denoted as 
\begin{align}\label{eq:gr-super}
 H=\bigoplus_{k\ge0}H^k=\bigoplus_{k\ge0}(H^k_{\ev} \oplus H^k_{\od}).
\end{align}

\item
$H$ is called \emph{graded connected} if it is graded, $H=\bigoplus_{k\ge0}H^k$,
and $H^0=\bfk$.

\item
A \emph{graded morphism} $H \to H'$ of Hopf superalgebras is a morphism of Hopf superalgebras
which preserves the $\bbN$-grading.
\end{enumerate}
\end{dfn}

\begin{rmk}
In this note, we will always denote the $\bbN$-grading of a graded Hopf superalgebra $H$
by the upper index $H^{*}$, and the parity ($\Zt$-grading) 
by the lower index $H_{\ol{*}}$ as in \eqref{eq:gr-super}.
\end{rmk}

For a graded Hopf superalgebra $H=\bigoplus_{k\ge0}H^k$, we denote the graded dual by 
$H^o \ceq \bigoplus_{k \ge 0}(H^k)^*$, $(H^k)^* \ceq \Hom(H^k,\bfk)$.
If $\dim H^k < \infty$ for each $k \ge 0$, then the graded dual $H^o$ 
has a natural Hopf superalgebra structure.
(See \cite[\S1.6]{GR} for the even case.
 For the existence of the antipode, see \cite[Theorem 2.1]{FLP}.)

\begin{dfn}\label{dfn:Hdual}
We call the Hopf superalgebra structure on $H^o$ the \emph{Hopf dual} of $H$.
\end{dfn}

\section{Symmetric and quasi-symmetric functions in superspace}\label{s:sLsQ}

Here we recall from \cite[\S\S3--5]{FLP} the Hopf superalgebras $\sL$ and $\sQ$ 
of symmetric and quasi-symmetric functions in superspace.
Let $\bfk$ be a field of characteristic $0$.

\subsection{Symmetric functions in superspace}\label{ss:sLsQ:sL}

Following \cite[\S2.2]{DLM}, \cite[\S2]{A+} and \cite[\S3]{FLP}, 
we denote by $\bfk[x,\theta]_N$ the polynomial algebra in the supervariables 
\[
 (x;\theta) = (x_1,\dotsc,x_N; \theta_1,\dotsc,\theta_N)
\]
with even $x$ and odd $\theta$, and consider the action of the symmetric group $\frS_N$ 
on $\bfk[x,\theta]_N$ in the diagonal way:  
\begin{align}\label{eq:sL:SN-act}
 \sigma(x_1,\dotsc,x_N; \theta_1,\dotsc,\theta_N) \ceq 
 (x_{\sigma(1)},\dotsc,x_{\sigma(N)}; \theta_{\sigma(1)},\dotsc,\theta_{\sigma(N)}) 
 \quad (\sigma \in \frS_N).
\end{align}
The invariant ring $\bfk[x,\theta]_N^{\frS_N}$ 
is called the ring of symmetric polynomials in $N$ supervariables. 
It is doubly graded: 
\begin{align}
 \bfk[x,\theta]_N^{\frS_N} = 
 \bigoplus_{n \ge 0} \bigoplus_{0 \le m \le N} (\bfk[x,\theta]_N^{\frS_N})_{n,m},  
\end{align}
where $(\bfk[x,\theta]_N^{\frS_N})_{n,m}$ is the space of homogeneous elements 
of degree $n$ in $x$ and of degree $m$ in $\theta$. 
The degree $n$ with respect to $x$ called the \emph{total degree},
and that with respect to $\theta$ called the \emph{fermionic degree}.

The space $\sL$ of symmetric functions in superspace
is defined by the inverse limit \cite[\S2.2]{A+}, \cite[\S3.1]{FLP}.
That is, for $M>N$, we consider a linear map 
$(\bfk[x,\theta]_M^{\frS_M})_{n,m} \to (\bfk[x,\theta]_N^{\frS_N})_{n,m}$ 
defined by $x_{N+1} = \dotsb = x_M=0$, $\theta_{N+1} = \dotsb = \theta_M=0$, 
while sending the other variables identically. 
Then we have an inverse system $\{ (\bfk[x,\theta]_N^{\frS_N})_{n,m} \}_N $ 
of linear spaces for each $n,m\ge0$.
Now we define a doubly graded linear space $\sL$ by 
\begin{align}
 \sL \ceq \bigoplus_{n,m \ge 0} \sL_{n,m}, \quad 
 \sL_{n,m} \ceq \varprojlim_N (\bfk[x,\theta]_N^{\frS_N})_{n,m},
\end{align}
and call it the space of symmetric functions in superspace.
For each $n,m \ge 0$, the subspace $\sL_{n,m}$ consists of symmetric polynomials 
in infinite supervariables of total degree $n$ and fermionic degree $m$. 
The space $\sL$ is an algebra which inherits the multiplication on $\bfk[x,\theta]_N^{\frS_N}$.

Using this bigrading, we can introduce an $\bbN$-graded superalgebra structure on $\sL$. 
First, define the parity ($\Zt$-grading) by 
\begin{align}\label{eq:sL:parity}
 \sL = \sL_{\ev}\oplus\sL_{\od}, \quad 
 \sL_{\ol{p}} \ceq \bigoplus_{n \ge 0, \, m \equiv \ol{p} \bmod 2}\sL_{n,m} \quad (p=0,1).
\end{align}
Then, $\sL$ is a commutative superalgebra (in the supercategory $\SVec$). 
Moreover, it has an additional $\bbN$-grading 
\begin{align}\label{eq:sL:N-grd}
 \sL = \bigoplus_{k \ge 0}\sL^k, \quad 
 \sL^k \ceq \bigoplus_{n+m=k} \sL_{n,m} \quad (k \ge 0), 
\end{align} 
which makes $\sL$ a graded commutative superalgebra. 

To explain the Hopf superalgebra structure on $\sL$, 
we recall from \cite[\S2.4, \S\S3.1--3.3]{DLM}, \cite[\S2.3]{A+} and \cite[\S3]{FLP}
some basic supersymmetric functions. 
We will freely use the notions and symbols of superpartitions 
in \cite[\S2.1]{DLM}, \cite[\S2.1]{A+} and \cite[\S3]{FLP}.
Let $\Lambda=(\Lambda^a;\Lambda^s)=(\Lambda_1,\dotsc,\Lambda_m;\Lambda_{m+1},\dotsc,\Lambda_N)$ 
be a superpartition of length $\le N$, total degree $n$ and fermionic degree $m$.
Then, the monomial polynomial $m_{\Lambda} \in (\bfk[x,\theta]_N^{\frS_N})_{n,m}$ is defined as
\begin{align}
 m_\Lambda(x,\theta) \ceq (\# \Aut(\Lambda^s))^{-1} 
 \sum_{\sigma \in \frS_N} \sigma(x^\lambda \theta_1\cdots\theta_m),
\end{align}
where $\Aut(\Lambda^s)$ denotes the order of the automorphism group of the partition 
$\Lambda^s=(\Lambda_{m+1},\dotsc,\Lambda_n)$ 
(see \cite[(2.39), (2.40)]{DLM} and \cite[(2.12)]{A+} for details),
and the $\frS_N$-action is given by \eqref{eq:sL:SN-act}.
The family $\{m_\Lambda \in (\bfk[x,\theta]_N^{\frS_N})_{n,m}\}_N$ is stable 
under the inverse limit, and we have the \emph{monomial symmetric function in superspace}
\begin{align}\label{eq:sL:mL}
 m_\Lambda \in \sL_{n,m}
\end{align}
for each superpartition $\Lambda$ of total degree $n$ and fermionic degree $m$.
The family $\{m_\Lambda \mid \text{$\Lambda$: superpartitions}\}$ is a basis of the space $\sL$.

Next, the \emph{elementary symmetric functions in superspace} $e_r,\wt{e}_s \in \sL$ are defined by 
\begin{align}\label{eq:sL:ee}
 e_r \ceq m_{(\emptyset; 1^r)}, \quad \wt{e}_s \ceq m_{(0; 1^s)} \quad (r \ge 1, \, s \ge 0),
\end{align}
Here we denoted $1^r \ceq \overbrace{1,\dotsc,1}^r$.
Then, the superalgebra $\sL$ is the polynomial superalgebra of $e_r$ and $\wt{e}_s$ 
\cite[Theorem 21]{DLM}, \cite[(2.20)]{A+}:
\begin{align}\label{eq:sL:e-gen}
 \Lambda \cong \bfk[e_1,e_2,\dotsc; \wt{e}_0,\wt{e}_1,\dotsc].
\end{align}
Below we use the additional symbol $e_0 \ceq 1 \in \sL$.

Now we can explain the comultiplication $\Delta$ on $\sL$ \cite[Proposition 4.1]{FLP}.
On the generator $e_r$ and $\wt{e}_s$, it is given by
\begin{align}
 \Delta(     e_r) = \sum_{k,l\ge0, \, k+l=r}      e_k \otimes e_l, \quad 
 \Delta(\wt{e}_s) = \sum_{k,l\ge0, \, k+l=s}(\wt{e}_k \otimes e_l+e_k \otimes \wt{e}_l). 
\end{align}
By \cite[Proposition 4.8]{FLP}, 
this comultiplication is cocommutative (in the supercategory $\SVec$)
and graded with respect to the $\bbN$-grading \eqref{eq:sL:N-grd}, 
and a superalgebra morphism $\sL \to \sL \otimes \sL$. 
Since $\sL$ is connected, i.e. $\sL^0=\sL^0_{\ev} \oplus \sL^0_{\od}=\bfk$, 
by \cite[Theorem 2.1]{FLP}, the bialgebra $\sL$ has a unique antipode, 
and thus is a Hopf superalgebra. 
Let us summarize as follows. 

\begin{prp}[{\cite[\S4]{FLP}}]\label{prp:sL}
The space $\sL$ of symmetric functions in superspace has 
a connected graded Hopf superalgebra structure which is commutative and cocommutative.
The parity is \eqref{eq:sL:parity} and the $\bbN$-grading is \eqref{eq:sL:N-grd}.
\end{prp}

\subsection{Quasi-symmetric functions in superspace} 

This subsection is based on \cite[\S5, \S6]{FLP}. 

Let us recall the notion of a \emph{dotted composition} from \cite[Definition 5.1]{FLP}.
It is a finite sequence $\alpha=(\alpha_1,\dotsc,\alpha_l)$ of entries belonging to 
the set $\{1,2,3,\dotsc\} \cup \{\dot{0},\dot{1},\dot{2},\dotsc\}$.
For such $\alpha$, we define $\eta(\alpha)=(\eta_1,\dotsc,\eta_l) \in \{0,1\}^l$ with 
\begin{align}\label{eq:sQ:eta}
 \eta_k \ceq \begin{cases}
  1 & (\text{if $\alpha_k$ is dotted}) \\ 0 & (\text{otherwise}) 
 \end{cases}.
\end{align}
The total degree of $\alpha$ is defined to be $\sum_{k=1}^l \alpha_k$, 
where the dotted entries $\alpha_k = \dot{n}$ are replaced by $n$.
The fermionic degree of $\alpha$ is defined to be the number of dotted entries in $\alpha$.

Let $(x;\theta)=(x_1,x_2,\dotsc; \theta_1,\theta_2,\dotsc)$ be the set of infinite supervariables
as before, and let $R$ be the ring of formal power series of finite degree in $(x;\theta)$. 
An element $f \in R$ is called a \emph{quasi-symmetric function in superspace} 
if for every dotted composition $\alpha=(\alpha_1,\dotsc,\alpha_l)$, all monomials 
$\theta^{\eta_1}_{i_1} \dotsm \theta^{\eta_l}_{i_l}x^{\alpha_1}_{i_1} \dotsm x^{\alpha_l}_{i_l}$ 
in $f$ with indices $i_1< \dotsb < i_l$ have the same coefficient, 
where $(\eta_1,\dotsc,\eta_l)=\eta(\alpha)$ is given by \eqref{eq:sQ:eta}.

Quasi-symmetric functions in superspace form a linear space, which is denoted by $\sQ$. 
It has a bigrading by total degree $n$ and fermionic degree $m$ \cite[(5.2)]{FLP}, 
which is denoted by 
\begin{align}
 \sQ = \bigoplus_{n,m\ge0}\sQ_{n,m}.
\end{align}
Then, $\sQ$ is a commutative superalgebra whose parity is the fermionic degree modulo 2:
\begin{align}\label{eq:sQ:parity}
 \sQ = \sQ_{\ev} \oplus \sQ_{\od}, \quad 
 \sQ_{\ol{p}} \ceq \bigoplus_{n \ge 0, \, m \equiv p \bmod 2} \sQ_{n,m} \quad (p=0,1).
\end{align}
It has an additional $\bbN$-grading: 
\begin{align}\label{eq:sQ:N-grd}
 \sQ = \bigoplus_{k\ge0} \sQ^k, \quad 
 \sQ^k \ceq \bigoplus_{n+m=k} \sQ_{n,m} \quad (k \ge 0).
\end{align}
One can see that the graded superspace $\sL$ of symmetric functions in superspace 
is a subspace of $\sQ$.

In order to explain the comultiplication $\Delta$ of $\sQ$,
let us introduce a natural basis \cite[\S5.1]{FLP}.
For each dotted composition $\alpha=(\alpha_1,\dotsc,\alpha_l)$, 
we define the \emph{monomial quasi-symmetric function} $M_{\alpha}$ by
\begin{align}\label{eq:sQ:M}
 M_{\alpha} \ceq \sum_{i_1<\dotsb<i_l} 
 \theta^{\eta_1}_{i_1} \dotsm \theta^{\eta_l}_{i_l}x^{\alpha_1}_{i_1} \dotsm x^{\alpha_l}_{i_l}. 
\end{align}
The family $\{M_{\alpha} \mid \text{$\alpha$: dotted compositions}\}$ is a basis of $\sQ$. 

Now, the comultiplication $\Delta$ is given in terms of $M_\alpha$, 
$\alpha=(\alpha_1,\dotsc,\alpha_l)$, as 
\begin{align}
 \Delta(M_{\alpha}) = 
 \sum_{k=0}^l M_{(\alpha_1,\dotsc,\alpha_k)} \otimes M_{(\alpha_{k+1},\dotsc,\alpha_{l})}. 
\end{align}
Note that $\Delta$ is not cocommutative (see \cite[Example 5.8]{FLP} for an explicit example).
By \cite[Proposition 5.9]{FLP}, $\Delta$ is graded with respect to the $\bbN$-grading 
\eqref{eq:sL:N-grd}, and is a superalgebra morphism $\sQ \to \sQ \otimes \sQ$. 
Since $\sQ$ is connected, i.e., $\sQ^0=\sQ^0_{\ol{0}} \oplus \sQ^0_{\ol{1}}=\bfk$, 
by \cite[Theorem 2.1]{FLP}, the bialgebra $\sQ$ has a unique antipode, 
and is a Hopf superalgebra. Let us summarize as follows. 

\begin{prp}[{\cite[\S 5]{FLP}}]\label{prp:sQ}
The space $\sQ$ of quasi-symmetric functions in superspace has 
a connected graded Hopf superalgebra structure which is commutative but not cocommutative. 
The parity is \eqref{eq:sQ:parity} and the $\bbN$-grading is \eqref{eq:sQ:N-grd}.
\end{prp}

\section{Combinatorial Hopf superalgebra}\label{s:CHSA}

We consider a superspace analogue of a combinatorial Hopf algebra \cite{ABS} 
(see also \cite[\S7]{GR}).

As before, let $\bfk$ be a field of characteristic $0$.
We denote by $\bfk[\ve]$ be the commutative superalgebra over $\bfk$ 
generated by an odd element $\ve$. 
Thus, we have $\ve^2=0$ and $\bfk[\ve]=\bfk+\bfk\ve$ as linear spaces.

\begin{dfn}\label{dfn:CHSA}
Let $H$ be a Hopf superalgebra over $\bfk$.
\begin{enumerate}
\item 
An even superalgebra morphism $\zeta\colon H \to \bfk[\ve]$  is called a \emph{supercharacter} of $H$.    

\item 
A \emph{combinatorial Hopf superalgebra} is a pair $(H,\zeta)$ consisting of 
a connected graded Hopf superalgebra $H$ and a supercharacter $\zeta\colon H \to \bfk[\ve]$.

\item 
Let $(H,\zeta)$ and $(H',\zeta')$ be combinatorial Hopf superalgebras.
A \emph{morphism} $(H,\zeta) \to (H',\zeta')$ of combinatorial Hopf superalgebras is 
a graded even morphism $\psi\colon H \to H'$ of Hopf superalgebras such that $\zeta=\zeta'\circ\psi$.
This yields the \emph{category of combinatorial Hopf superalgebras}.
\end{enumerate}
\end{dfn}

Note that, for a supercharacter $\zeta\colon H \to \bfk[\ve]$, 
we have $\zeta(1)=1$, $\abs{\zeta(a)}=\abs{a}$ and $\zeta(ab)=\zeta(a)\zeta(b)$ for $a,b \in H$.
In particular, for $a,b \in H_{\od}$, we have $\zeta(a),\zeta(b)\in \bfk\ve$ and $\zeta(ab)=0$.


For example, the Hopf superalgebra $\sQ$ has a supercharacter 
\begin{align}\label{eq:zeta_Q}
 \zeta_Q\colon \sQ \lto \bfk[\ve]
\end{align}
given by the specialization $x_1=1$, $x_2=x_3=\dotsb=0$ and 
$\theta_1=\ve$, $\theta_2=\theta_3=\dotsb=0$. 
The same specialization defines a supercharacter on the Hopf superalgebra $\sL$:
\begin{align}\label{eq:zeta_S}
 \zeta_S\colon \sL \lto \bfk[\ve].
\end{align}
Then, by \cref{prp:sL,prp:sQ},
the pairs $(\sQ,\zeta_Q)$ and $(\sL,\zeta_S)$ are combinatorial Hopf superalgebras.

Another example of a supercharacter on $\sQ$ is given by 
the specialization $x_i=1$, $\theta_i=\ve$ and $x_j=\theta_j=0$ ($j \ne i$)
for some fixed $i$. The case $i=1$ gives $\zeta_S$.
Note that the same specialization gives $\zeta_S$ on $\sL$ for any $i$.

As will be shown in \cref{thm:CHSA:sQ}, \cref{prp:CHSA:psi} and \cref{prp:CHSA:sL} below, 
the supercharacters $\zeta_Q$ and $\zeta_S$ are natural from the perspective of 
the structure theory of the category of combinatorial Hopf superalgebras.

Now we come to the main statement of this note,
which is a superspace extension of \cite[Theorem 7.1.3]{GR}.

\begin{thm}\label{thm:CHSA:sQ}
The pair $(\sQ,\zeta_Q)$ is a terminal object 
of the category of combinatorial Hopf superalgebras.
In other words, we have: \\
Let $A=\bigoplus_{k\ge0}A^k$ be a combinatorial Hopf superalgebra 
with supercharacter $\zeta\colon A \to \bfk[\ve]$.
Then, there is a unique graded even morphism $\Psi\colon A \to \sQ$ of graded Hopf superalgebras 
such that $\zeta=\zeta_Q\circ\Psi$, where $\zeta_Q$ is the supercharacter \eqref{eq:zeta_Q}.
\end{thm}


For the proof, we need the Hopf superalgebra $\sN$ of 
noncommutative symmetric functions in superspace 
(see \cite[\S6]{FLP} and \cite{AGM} for details).
It is defined to be the Hopf dual of $\sQ$ (c.f.\ \cref{dfn:Hdual}).
Precisely speaking, its underlying $\bfk$-linear space is the bigraded dual of $\sQ$, i.e.,  
$\sN \ceq \bigoplus_{n,m\ge0}\sN_{n,m}$ with $\sN_{n,m}\ceq(\sQ_{n,m})^*$.
Also, $\sN$ has a $\bfk$-basis $\{H_\alpha\}$ labeled by dotted compositions $\alpha$
which is dual to the family $\{M_\alpha\}$ of 
monomial quasi-symmetric functions in superspace \eqref{eq:sQ:M}.
By \cite[Proposition 6.2]{FLP}, $\sN$ is isomorphic (as an algebra)  
to the free associative algebra with noncommuting generators 
$\{H_1,H_2,\dotsc,;\wt{H}_0,\wt{H}_1,\dotsc\}$, where we set 
\begin{align}\label{eq:CHSA:HrwHs}
 H_r \ceq H_{(r)} \quad (r \ge 1), \quad \wt{H}_s \ceq H_{(\dot{s})} \quad (s \ge 0).   
\end{align}
The Hopf superalgebra $\sN$ is connected graded for the grading 
$\sN=\bigoplus_{k\ge0}\sN^k$ with $\sN^k \ceq \bigoplus_{n+m=k}\sN_{n,m}$.
Note that we have $H_r \in \sN_{r,0} \subset \sN^r_{\ev}$ ($r\ge1$) 
and $\wt{H}_s \in \sN_{s,1} \subset \sN^{s+1}_{\od}$ ($s\ge0$). 

Let us also remark that giving a supercharacter on $\sQ$ is equivalent to 
giving a family of functionals in each homogeneous space 
$\sQ^k_{\ol{p}} = \bigoplus_{n+m=k, \, m \equiv \ol{p} \bmod 2}\sQ_{n,m}$. 
Under this equivalence, the supercharacter $\zeta_Q$ corresponds to the functionals $H_r,\wt{H}_s$.
That is, for $f \in \sQ^k$, we have
\begin{align}\label{eq:sQ:func}
 \zeta_Q(f) = 
 \begin{cases}(H_k, f) & (\abs{f}=\ev) \\ (\wt{H}_{k-1}, f) & (\abs{f}=\od) \end{cases},
\end{align}
where $\abs{p}$ denotes the parity of $f$.

\begin{proof}[Proof of \cref{thm:CHSA:sQ}]
We follow the non-super case \cite[Theorem 7.1.3]{GR}. 
First, note that we can decompose the given $\zeta\colon A \to \bfk[\ve]$ 
as $\zeta=\zeta_{\ev} \oplus \zeta_{\od}\ve$, 
where $\zeta_{\ol{p}}\colon A_{\ol{p}} \to \bfk$ ($p=0,1$) are linear, 
$\zeta_{\ev}$ is even and $\zeta_{\od}$ is odd.
Let us define an even linear map 
\begin{align}\label{eq:zeta'}
 \zeta' \colon \left(\tboplus_{r\ge1} \bfk      H_r\right) \oplus 
               \left(\tboplus_{s\ge0} \bfk \wt{H}_s\right) 
               \lto A^o = A^o_{\ev} \oplus A^o_{\od},
\end{align}
where $A^o \ceq \bigoplus_{k\ge0}(A^k)^*$ is the graded dual 
of the graded superspace $A=\bigoplus_{k\ge0}A^k$, by
\begin{align}
\label{eq:CHSA:zeta'Hr}
 \zeta'(H_r)(a) &\ceq 
 \begin{cases}
 \zeta_{\ev}(a)& (a \in A^r_{\ev})\\
 0 & (\text{$a$ is in the other homogeneous spaces})
 \end{cases},
\\
\label{eq:CHSA:zeta'wHs}
 \zeta'(\wt{H}_s)(a) &\ceq 
 \begin{cases}
 \zeta_{\od} (a) & (a \in A^{s+1}_{\od} ), \\
 0 & (\text{$a$ is in the other homogeneous spaces})
 \end{cases}.
\end{align}
Since $\sQ^o=\sN$ is the free associative superalgebra 
$\bfk\langle H_1,H_2,\dotsc;\wt{H}_0,\wt{H}_1,\dotsc\rangle$, 
the map $\zeta'$ extends uniquely to an even morphism $\Psi^* \colon \sQ^o \to A^o$
of superalgebras. 
Moreover, the definition of $\zeta'$ implies that $\Psi^*$ preserves the $\bbN$-grading. 
Since each homogeneous subspace $\sQ^k$ ($k\in\bbN$) is finite dimensional,  
there exists a unique graded even coalgebra morphism $\Psi \colon A \to \sQ$ 
which is a graded dual of $\Psi^*$. 

Then, for $a \in A^r_{\ev}$ ($r\ge 1$), we have 
\begin{align}
 (\zeta_Q \circ \Psi)(a) = H_r(\Psi(a)) = \Psi^*(H_r)(a) = \zeta_{\ev}(a) = \zeta(a),
\end{align}
where the first equality follows from \eqref{eq:sQ:func}. 
Similarly, for $a \in A^{s+1}_{\od}$ ($s \ge 0$), we have 
\begin{align}
 (\zeta_Q \circ \Psi)(a) = \wt{H}_s(\Psi(a))\ve = 
 \Psi^*(\wt{H}_s)(a)\ve  = \zeta_{\od}(a)\ve = \zeta(a). 
\end{align}
Thus, we have a unique graded even coalgebra morphism $\Psi$ 
such that $\zeta=\zeta_Q \circ \Psi$.

It remains to show that $\Psi$ is an algebra morphism. 
Consider the following two diagrams: 
\begin{align}
\begin{tikzcd}[ampersand replacement=\&]
 A^{\otimes 2} \ar[r, "\mu^{\otimes 2}"] \ar[dr, "\zeta^{\otimes 2}"'] \& 
 A \ar[r, "\Psi"] \ar[d, "\zeta"] \& 
 \sQ \ar[dl, "\zeta_Q"] \& 
 A^{\otimes 2} \ar[r, "\Psi^{\otimes 2}"] \ar[dr, "\zeta^{\otimes 2}"'] \& 
 \sQ^{\otimes 2} \ar[r, "\mu"]  \ar[d, "\zeta_Q^{\otimes 2}"] \&  
 \sQ \ar[dl, "\zeta_Q"] \\
 \& \bfk \oplus \bfk \ve \& \& 
 \& \bfk \oplus \bfk \ve 
\end{tikzcd}
\end{align}
Since both diagrams commute and both $\Psi \circ (\mu \otimes \mu)$ 
and $\mu \circ (\Psi \otimes \Psi)$ are coalgebra morphisms, 
the uniqueness proved above implies that 
$\Psi \circ (\mu \otimes \mu)=\mu \circ (\Psi \otimes \Psi)$. 
The proof is now completed.
\end{proof}

\begin{rmk*}
The Hopf superalgebra $\sN$ admits a supercharacter $\zeta_N$ defined by the specialization 
on generators: $\zeta_N(H_1) = 1$, $\zeta_N(H_2)=\zeta_N(H_3)=\dotsb=0$,
$\zeta_N(\wt{H}_0)=\ve$ and $\zeta_N(\wt{H}_1)=\zeta_N(\wt{H}_2)=\dotsb=0$.
\end{rmk*}

We have the following formula of the image $\Psi(a)$ of $a \in A$ in \cref{thm:CHSA:sQ},
which is a superspace extension of \cite[(7.1.3)]{GR}.

\begin{prp}\label{prp:CHSA:psi}
Let $(A,\zeta)$ and $\Psi\colon A \to \sQ$ be the same as in \cref{thm:CHSA:sQ}.
Then, for a homogeneous element $a \in A^k$, $k \ge 0$, we have 
\begin{align}\label{eq:CHSA:Psi-sQ}
 \Psi(a) = \sum_{\text{$\alpha$: dotted compositions of $k$}} \zeta_\alpha(a) M_\alpha,
\end{align}
where, for a dotted composition $\alpha=(\alpha_1,\dotsc,\alpha_l)$, 
the symbol $M_\alpha$ denotes the monomial quasi-symmetric function \eqref{eq:sQ:M}, 
and the map $\zeta_\alpha$ is the composite 
\begin{align}\label{eq:CHSA:zeta_alpha}
 \zeta_\alpha\colon 
 A^k \xrr{ \Delta^{(l-1)} } A^{\otimes l} \xrr{\pi_\alpha} 
 A^{\alpha_1}_{\ol{\eta_1}} \otimes \dotsb \otimes A^{\alpha_l}_{\ol{\eta_l}}
 \xrr{\zeta^{\otimes l}} \bfk.
\end{align}
Here $\eta(\alpha)=(\eta_1,\dotsc,\eta_l)$ is given in \eqref{eq:sQ:eta},
the dotted superscript $\alpha_k=\dot{n}$ of $A^{\alpha_k}_{\ol{\eta_k}}$ 
are replaced by $n$,
and the symbol $\pi_\alpha$ is the natural projection.
\end{prp}

\begin{proof}
By the proof of \cref{thm:CHSA:sQ}, we have for $a \in A^k$ that 
\begin{align}
 \Psi(a) = \sum_{\text{$\alpha$: dotted composition of $k$}} (H_\alpha,\Psi(a)) M_\alpha,
\end{align}
and for each dotted superscript $\alpha=(\alpha_1,\dotsc,\alpha_l)$, we have  
\begin{align}
 (H_\alpha,\Psi(a)) = (\Psi^*(H_\alpha),a) = 
 \bigl(\Psi^*(H_{(\alpha_1)}) \dotsm \Psi^*(H_{(\alpha_l)}),a\bigr) = 
 (\zeta^{\otimes l}\circ\pi_\alpha)\bigl(\Delta^{(l-1)}(a)\bigr) = \zeta_\alpha(a).
\end{align}
Here  we used the equality $H_\alpha=H_{(\alpha_1)}\dotsm H_{(\alpha_l)}$ \cite[(6.4)]{FLP}
in the second equality, and we used \eqref{eq:CHSA:HrwHs}, \eqref{eq:CHSA:zeta'Hr} 
and \eqref{eq:CHSA:zeta'wHs} in the third equality.
\end{proof}

Now we turn to the Hopf superalgebra $\sL$ of symmetric functions in superspace (\cref{ss:sLsQ:sL}).
It is a combinatorial Hopf superalgebra with supercharacter $\zeta_S$ \eqref{eq:zeta_S},
and is cocommutative.
We have the following characterization.

\begin{prp}\label{prp:CHSA:sL}
The pair $(\sL,\zeta_S)$ is a terminal object of 
the category of cocommutative combinatorial Hopf superalgebras.
In other words, we have: \\
Let $A$ be a cocommutative combinatorial Hopf superalgebra 
with supercharacter $\zeta\colon A \to \bfk[\ve]$.
Then, there is a unique graded even morphism $\Psi\colon A \to \sL$ of Hopf superalgebras
such that $\zeta=\zeta_S\circ\Psi$, 
where $\zeta_S\colon \sL \to \bfk[\ve]$ is the supercharacter \eqref{eq:zeta_S}.

Moreover, for a homogeneous element $a \in A^k$, $k \ge 0$, we have 
\begin{align}\label{eq:CHSA:Psi-sL}
 \Psi(a) = \sum_{\text{$\Lambda$: superpartitions of $k$}} \zeta_\Lambda(a)m_\Lambda,
\end{align}
where, for a superpartition $\Lambda=(\Lambda_1,\dotsc,\Lambda_m;\Lambda_{m+1},\dotsc,\Lambda_l)$, 
the symbol $m_\Lambda$ denotes the monomial symmetric function in superspace \eqref{eq:sL:mL},
and the symbol $\zeta_\Lambda$ denotes the composite \eqref{eq:CHSA:zeta_alpha} 
for the dotted composition $(\dot{\Lambda}_1,\dotsc,\dot{\Lambda}_m,\Lambda_{m+1},\dotsc,\Lambda_l)$.
\end{prp}

\begin{proof}
It is enough to slightly modify the proof of \cref{thm:CHSA:sQ}. 
Note first that the graded dual $A^o$ is a commutative superalgebra
since $A$ is now a cocommutative co(-super)algebra.
Then, recalling the description 
$\sL \cong \bfk[e_1,e_2,\dotsc;\wt{e}_0,\wt{e}_1,\dotsc]$ \eqref{eq:sL:e-gen}, 
we modify \eqref{eq:zeta'} and define an even linear map 
\begin{align} 
 \zeta''\colon &\left(\tboplus_{r\ge1} \bfk      e_r\right) \oplus 
                \left(\tboplus_{s\ge0} \bfk \wt{e}_s\right) 
                \lto A^o = A^o_{\ev} \oplus A^o_{\od}, \\
 \zeta''(e_r)(a) &\ceq 
 \begin{cases}
 \zeta_{\ev}(a)& (a \in A^r_{\ev})\\
 0 & (\text{$a$ is in the other homogeneous spaces})
 \end{cases}, \\
 \zeta''(\wt{e}_s)(a) &\ceq 
 \begin{cases}
 \zeta_{\od} (a) & (a \in A^{s+1}_{\od} ), \\
 0 & (\text{$a$ is in the other homogeneous spaces})
 \end{cases}.
\end{align}
Since $\sL$ is the polynomial superalgebra generated 
by the $e_r$ and the $\wt{e}_s$ \eqref{eq:sL:e-gen},
the map $\zeta''$ extends uniquely to an even morphism $\Psi^*\colon \sL \to A^o$ 
which preserves the $\bbN$-grading.
Then, since $\sL \cong \sL^o$, i.e.,  $\sL$ is self-dual \cite[Proposition 4.8]{FLP},
we have a graded even morphism $\Psi\colon A \to \sL^o \cong \sL$. 
The remaining parts of the statement can be shown as in \cref{thm:CHSA:sQ}.
\end{proof}

\section{A superspace analogue of chromatic symmetric functions}\label{s:chr}

We consider a superspace analogue of 
the Hopf algebra of Stanley's chromatic symmetric functions \cite[\S7.3]{GR}.

Let us denote a finite graph by $(V,E)$, where $V$ is the vertex set and $E$ is the edge set
(both are assumed to be finite).
Consider a triple $G=(V,W,E)$ consisting of a finite graph $(V,E)$ and a subset $W \subset V$,
equipped with an order of the connected components of $(V,E)$.
Such a triple can be regarded as a two-colored graph with white vertices $W$ and 
black vertices $V \setminus W$.
We call $G$ connected if $(V,E)$ is connected.
We depict such a triple by white vertices $W$ and black vertices $V \setminus W$,
and inserting a bar $|$ among the ordered connected components.
For example, the diagram below shows a triple $G=(V,W,E)$ with $\# V=7$, $\# W=3$ 
and the number of connected components is $2$.
\begin{align}
 \bl - \bl - \circ | \bl - \circ - \circ - \bl
\end{align}
We also denote by $\emptyset$ the empty triple.

Let $\clG'$ be the $\bfk$-linear space spanned by $[G]$ corresponding to each triple $G$, 
and $\clC \subset \clG'$ be the subspace spanned by the elements
\begin{align}
 [G_1|G_2] - (-1)^{\#W_1 \cdot \# W_2} [G_2|G_1] \quad (G_i=(V_i,W_i,E_i), \text{connected}).
\end{align}
We define $\clG$ to be the quotient space 
\begin{align}
 \clG \ceq \clG' / \clC,    
\end{align}
and denote the class in $\clG$ of $[G] \in \clG'$ by the same symbol $[G]$.
In particular, the following holds in $\clG$.
\begin{align}\label{eq:chr:subsp}
 [\circ|\circ] = 0 
\end{align}

The space $\clG$ has a double grading 
\begin{align}
 \clG = \bigoplus_{n,m\ge0}\clG_{n,m},
\end{align}
where $\clG_{n,m}$ is spanned by $[G]$ with $G=(V,W,E)$, $\# V=n$ and $\# W=m$.
It induces an $\bbN$-grading: 
\begin{align}\label{eq:chr:clG^*}
 \clG = \bigoplus_{k\ge0}\clG^k, \quad 
 \clG^k \ceq \bigoplus_{n+m=k}\clG_{n,m}.
\end{align}
Note that we have 
\begin{align}\label{eq:chr:clG^0}
 \clG^0 = \bfk = \bfk[\emptyset].
\end{align}
We make $\clG$ into a superspace by defining the $\Zt$-grading to be $m \bmod 2$:
\begin{align}
 \clG = \clG_{\ev} \oplus, \clG_{\od}, \quad 
 \clG_{\ol{p}} \ceq \bigoplus_{n\ge0, \, m \bmod 2 = \ol{p}} \clG_{n,m} \quad (p=0,1).
\end{align}

With these definitions, we see that the subspace
\begin{align}\label{eq:chr:Gn0}
 \bigoplus_{n\ge0}\clG_{n,0} \subset \clG
\end{align}
is equal to the underlying linear space of the chromatic Hopf algebra 
in \cite[Definition 7.3.1]{GR}.    

The superspace $\clG$ has the following structure of a connected graded Hopf superalgebra,
which is a natural super analogue of \cite[Definition 7.3.1]{GR}.
The multiplication is 
\begin{align}
 [G_1] \cdot [G_2] \ceq [G_1|G_2].
\end{align}
Then $\clG$ is a commutative superalgebra whose unit is $[\emptyset]$,
and generated by $[G_{\con}]$ with connected triples $G_{\con}$.
Moreover, the subalgebra $\bigoplus_{n\ge0}\clG_{n,0} \subset \clG$ 
is equal to the non-super case \cite[\S7.3]{GR}.
For example, 
\begin{align}
 [\circ] \cdot [\circ] = 0
\end{align}
holds by \eqref{eq:chr:subsp}, but $[\bl|\bl] \ne 0$ in $\clG$.

To explain the comultiplication, we give some preliminary.
For a triple $G=(V,W,E)$ and a subset $V' \subset V$,
we define a new triple $\rst{G}{V'} = (V',W',E')$ by 
$W' \ceq W \cap V'$, $E' \ceq \{e\in E \mid e=\{v_1,v_2\} \subset V'\}$
and the order of the connected components $\rst{G}{V'}=G_1|\cdots|G_l$, 
$G_i=(V_i,W_i,E_i)$ being such that $\# W_1 \le \cdots \le \# W_l$
(such an order is not unique, but the equivalence class in $\clG$ is uniquely determined). 
Now, for each connected triple $G_{\con}$, we define 
\begin{align}
 \Delta([G_{\con}]) \ceq \frac{1}{2}\sum_{(V_1,V_2): V_1 \sqcup V_2 = V} 
 \left([\rst{G_{\con}}{V_1}] \otimes [\rst{G_{\con}}{V_2}] + 
       (-1)^{\#W_1\#W_2}[\rst{G_{\con}}{V_2}]\otimes[\rst{G_{\con}}{V_1}]\right).
\end{align}
For an arbitrary triple $G$, we decompose $G = G_1|\cdots|G_l$ into connected $G_i$'s and define 
\begin{align}
 \Delta([G]) \ceq \Delta([G_1])\cdots\Delta([G_l]).
\end{align}
For example, we have
\begin{align}
 \Delta([\bl-\bl]) &= [\bl-\bl]\otimes 1 + 2[\bl]\otimes[\bl] + 1\otimes[\bl-\bl],
\\
 \Delta([\bl-\circ]) &= 
 [\bl-\circ]\otimes 1 + [\bl]\otimes[\circ] + [\circ]\otimes[\bl] + 1\otimes[\bl-\circ],
\\
 \Delta([\circ-\circ]) &= [\circ-\circ]\otimes 1 + 1\otimes[\circ-\circ],
\\
 \Delta([\circ-\circ-\bl]) &= 
 [\circ-\circ-\bl]\otimes 1 + [\bl]\otimes[\circ-\circ] + 
 [\circ-\circ]\otimes[\bl] + 1\otimes[\circ-\circ-\bl], 
\\
 \Delta([\begin{tikzpicture}[scale=0.5,baseline=-3pt]
   \draw[-] (4.0,0) node {$\bl$} -- (4.9,0);
   \node at (5.0,0) {$\circ$};
   \draw[-] (5.15,0.05) -- (6,0.3) node {$\bl$} 
   (6, -0.3) node {$\bl$} -- (5.15,-0.05); \end{tikzpicture}]) &= 
 [\begin{tikzpicture}[scale=0.5,baseline=-3pt]
   \draw[-] (4.0,0) node {$\bl$} -- (4.9,0);
   \node at (5.0,0) {$\circ$};
   \draw[-] (5.15,0.05) -- (6,0.3) node {$\bl$} 
   (6, -0.3) node {$\bl$} -- (5.15,-0.05); \end{tikzpicture}] \otimes 1 +
 1\otimes[\begin{tikzpicture}[scale=0.5,baseline=-3pt]
   \draw[-] (4.0,0) node {$\bl$} -- (4.9,0);
   \node at (5.0,0) {$\circ$};
   \draw[-] (5.15,0.05) -- (6,0.3) node {$\bl$} 
   (6, -0.3) node {$\bl$} -- (5.15,-0.05); \end{tikzpicture}] + 
 3[\bl]\otimes[\bl-\circ-\bl] + 3[\bl-\circ-\bl]\otimes[\bl] \\
 & \quad +
 3[\bl-\circ]\otimes[\bl|\bl] +
 3[\bl|\bl]\otimes[\bl-\circ] + 
  [\circ]\otimes[\bl|\bl|\bl] + [\bl|\bl|\bl]\otimes[\circ].
\end{align}
The map $\Delta$ is coassociative and a superalgebra morphism, 
and we call it the comultiplication of $\clG$.
Moreover, it is cocommutative. 
Indeed, for any $G = (V, W, E)\in\clG$, we have
\begin{align}
  \tau\circ\Delta([G])
&=\frac{1}{2}\sum_{(V_1,V_2): V_1 \sqcup V_2 = V}
  \left(\tau([\rst{G}{V_1}] \otimes [\rst{G}{V_2}]) + 
        (-1)^{\#W_1\#W_2}\tau([\rst{G}{V_2}] \otimes [\rst{G}{V_1}])\right) \\
&=\frac{1}{2}\sum_{(V_1,V_2): V_1 \sqcup V_2 = V}
  \left((-1)^{\#W_1\#W_2}[\rst{G}{V_2}] \otimes [\rst{G}{V_1}] + 
        [\rst{G}{V_1}] \otimes [\rst{G}{V_2}]\right)
 = \Delta([G]),
\end{align}
where $\tau$ denotes the switching: for $G_i=(V_i,W_i,E_i)$, $i=1,2$, 
\begin{align}
 \tau\colon: \clG \otimes \clG \lto \clG \otimes \clG, \quad 
 [G_1] \otimes [G_2] \lmto (-1)^{\# W_1 \cdot \# W_2} [G_2] \otimes [G_1].
\end{align}

The restriction of the comultiplication $\Delta$ to 
the subspace $\bigoplus_{n\ge0}\clG_{n,0} \subset \clG$ \eqref{eq:chr:Gn0}
is equal to the non-super case \cite[Definition 7.3.1]{GR}.
Indeed, for a triple $G=(V,\emptyset,E)$, we have
\begin{align}
  \Delta([G]) 
&=\frac{1}{2}\sum_{(V_1,V_2): V_1 \sqcup V_2 = V} 
  \left([\rst{G}{V_1}] \otimes [\rst{G}{V_2}] + [\rst{G}{V_2}]\otimes[\rst{G}{V_1}]\right) \\
\label{eq:chr:D|}
&=\sum_{(V_1,V_2): V_1 \sqcup V_2 = V} [\rst{G}{V_1}] \otimes [\rst{G}{V_2}].
\end{align}

From the discussion so far, $\clG$ is a commutative and cocommutative bi(-super)algebra.
By \eqref{eq:chr:clG^*} and \eqref{eq:chr:clG^0}, it is connected graded.
Then, by \cite[Theorem 2.1]{FLP} and \cite[Proposition 1.4.16]{GR} 
(using the $\bbN$-grading), $\clG$ has an antipode (it is automatically unique).
The obtained Hopf superalgebra $\clG$ is connected graded 
by \eqref{eq:chr:clG^*} and \eqref{eq:chr:clG^0}.
We can also see that the subspace \eqref{eq:chr:Gn0} is a connected graded Hopf subalgebra of $\clG$,
and coincides with the \emph{chromatic Hopf algebra} in \cite[Definition 7.3.1]{GR}.

\begin{dfn}
The connected graded Hopf superalgebra $\clG$ is called the \emph{chromatic Hopf superalgebra}.
\end{dfn}

Moreover, the Hopf superalgebra $\clG$ has a supercharacter $\zeta_{\ch}$.
To specify it, it is enough to give the value on each connected triple $G_{\con}=(V,W,E)$.
We set 
\begin{align}\label{eq:chr:efz0}
 \zeta_{\ch}([G_{\con}]) \ceq 
 \begin{cases} 
  1   & \text{if $E = \emptyset$ and $\# W = 0$} \\ 
  \ve & \text{if $E = \emptyset$ and $\# W = 1$} \\
  0 & \text{otherwise}
 \end{cases}
\end{align}
For an arbitrary triple $G=G_1|\cdots|G_l$ with connected $G_i$'s, 
we set $\zeta_{\ch}([G]) \ceq \zeta_{\ch}([G_1])\cdots\zeta_{\ch}([G_l])$, and obtain a supercharacter 
\begin{align}\label{eq:chr:efz}
 \zeta_{\ch}\colon \clG \lto \bfk[\ve].   
\end{align}
If we restrict $\zeta_{\ch}$ to the Hopf subalgebra $\bigoplus_{n\ge0}\clG_{n,0} \subset \clG$,
then we recover the \emph{edge-free character} in \cite[Definition 7.3.16, (7.3.4)]{GR}. 

In summary, we have:

\begin{thm}\label{thm:chr:clG}
The chromatic Hopf superalgebra $\clG$ with the supercharacter $\zeta_{\ch}$ in \eqref{eq:chr:efz}
is a combinatorial Hopf superalgebra which is both commutative and cocommutative.
It contains the chromatic Hopf algebra $\bigoplus_{n\ge0}\clG_{n,0}$ in \cite[Definition 7.3.1]{GR}
as a Hopf subalgebra.  
\end{thm}

Since $(\clG,\zeta_{\ch})$ is a cocommutative combinatorial Hopf superalgebra, by \cref{prp:CHSA:sL}, 
there is a graded even morphism 
\begin{align}
 \Psi_{\ch}\colon \clG \lto \sL   
\end{align}
of Hopf superalgebras.
Restricted to the subalgebra $\bigoplus_{n\ge0}\clG_{n,0}$, it recovers the map to 
the space of symmetric functions for the non-super case \cite[Definition 7.3.16]{GR}.
Since the image of this non-super map consists of 
\emph{Stanley's chromatic symmetric functions of graphs}, 
we may introduce:

\begin{dfn}\label{dfn:chr:PsiG}
For each triple $G=(V,W,E)$, we call the image $\Psi_{\ch}([G]) \in \sL$ 
the \emph{chromatic symmetric function in superspace} associated to $G$.
\end{dfn}

Let us give some examples of chromatic symmetric functions in superspace.
\begin{align}
&\Psi_{\ch}([\bl])=e_1, \quad 
 \Psi_{\ch}([\bl-\bl])=2e_2, 
\\
&\Psi_{\ch}([\circ])=\wt{e}_0, \quad 
 \Psi_{\ch}([\bl - \circ]) = \wt{e}_1, \quad
 \Psi_{\ch}([\circ - \circ - \bl]) = 0, 
\\ 
&\Psi_{\ch}([\begin{tikzpicture}[scale=0.5,baseline=-3pt]
   \draw[-] (4.0,0) node {$\bl$} -- (4.9,0);
   \node at (5.0,0) {$\circ$};
   \draw[-] (5.15,0.05) -- (6,0.3) node {$\bl$} 
   (6, -0.3) node {$\bl$} -- (5.15,-0.05); \end{tikzpicture}])
 = m_{(\dot{0},3)} + 6m_{(\dot{0},1,1,1)} + 3m_{(\dot{0},2,1)},
\end{align}
where we used the monomial symmetric function in superspace \eqref{eq:sL:mL} 
and the elementary symmetric functions in superspace \eqref{eq:sL:ee}.
From the definition \eqref{eq:chr:efz0}, \eqref{eq:chr:efz}, we have:

\begin{lem}\label{lem:chr:PsiC=0}
If a triple $G=G_1|\cdots|G_l$ has a connected component $G_k=(V_k,W_k,E_k)$ with $\#W_k>1$,
then $\Psi_{\ch}([G])=0$.
\end{lem}

Let us give another example of a chromatic symmetric function in superspace,
which can be regarded as a superspace analogue of \cite[Example 7.3.18]{GR}.
We will denote the complete graph on $l$ vertices by $K_l$.

\begin{prp}\label{prp:chr:Kn1}
For $n \ge 0$, let $K_{n+1,1}=(V,W,E)$ be a triple 
such that $(V,E)$ is the complete graph $K_{n+1}$ and $\#W=1$. 
Then, using the elementary symmetric function $\wt{e}_n$ \eqref{eq:sL:ee}, we have 
\begin{align}
 \Psi_{\ch}([K_{n+1,1}]) = n! \, \wt{e}_n.
\end{align}  
\end{prp}

\begin{proof}
Let us once admit the equality
\begin{align}\label{eq:chr:Dn}
 \Delta^{(n)}([K_{n+1, 1}]) &= n!\bigl(
  [\circ]\otimes[\bl]^{\otimes n} + 
  [\bl]\otimes[\circ]\otimes[\bl]^{\otimes(n-1)} + \cdots + 
  [\bl]^{\otimes n}\otimes[\circ]\bigr) + O,
\end{align}
where the last term $O \in \clG^{\otimes n+1}$ satisfies $\zeta_{\ch}^{\otimes (n+1)}(O)=0$.
For a dotted partition $\alpha$, we denote by $\zeta_{\ch,\alpha}$ 
the composite \eqref{eq:CHSA:zeta_alpha} for the supercharacter $\zeta_{\ch}$.
Then, for each $i=1,\dotsc,n+1$, we have
\begin{align}
 \zeta_{\ch,\alpha}([K_{n+1,1}]) = 
 \begin{cases} n! & (\alpha=(1^{i-1},\dot{0},1^{n+1-i})) \\ 
               0  & (\text{otherwise}) \end{cases}.
\end{align}
Hence, by \eqref{eq:CHSA:Psi-sL} in \cref{prp:CHSA:sL}, we have 
\begin{align}
 \Psi_{\ch}([K_{n+1,1}]) = n! \, m_{(0;1^n)} = n! \, \wt{e}_n    
\end{align}

We will show \eqref{eq:chr:Dn} in a more general form. 
For any $2\le m\le n$, we have
\begin{align}
\begin{split}
\label{eq:chr:DKn1}
\Delta^{(m)}([K_{n+1,1}])
&=   n(n-1)\cdots(n-m+1)[K_{n-m+1,1}] \otimes [\bl]^{\otimes m} \\
&\  + n(n-1)\cdots(n-m+2)[K_{n-m+1}] \otimes 
      \sum_{k=1}^m [\bl]^{\otimes (k-1)} \otimes [\circ] \otimes [\bl]^{\otimes (m-k)}
    + O
\end{split}
\end{align}
where $K_{l}=(V_l,\emptyset,E_l)$ is the triple 
corresponding to the ordinary complete graph on $l$ vertices, 
and the term $O$ satisfies $\zeta_{\ch}^{\otimes (m+1)}(O)=0$.
Note that $O$ contains terms of the form $[G_0] \otimes [G_1] \otimes \cdots \otimes [G_m]$
such that for some $k=0,\dotsc,m$ we have $G_k=(V_k,W_k,E_k)$ with $\# E_k>0$. 
If $m=n$, then \eqref{eq:chr:DKn1} implies \eqref{eq:chr:Dn}.

As a preliminary of proving \eqref{eq:chr:DKn1}, 
we calculate $\Delta(K_l)$ using \eqref{eq:chr:D|}. We have
\begin{align}
 \Delta([K_l]) = [K_l] \otimes 1 + l [K_{l-1}] \otimes [\bl] + O,
\end{align}
where the term $O$ means the same as \eqref{eq:chr:DKn1}.
In particular, for $l>1$, we have 
\begin{align}\label{eq:chr:DKl}
 \Delta([K_l])=l [K_{l-1}] \otimes [\bl] + O.
\end{align}
Similarly, we have 
\begin{align}\label{eq:chr:DKl1}
\begin{split}
  \Delta([K_{l+1,1}]) 
&=\frac{1}{2}\sum_{\#V_1=1, \, \#V_2=l}
  \left([K_{l+1,1}|_{V_1}]\otimes[K_{l+1,1}|_{V_2}] + 
        [K_{l+1,1}|_{V_2}]\otimes[K_{l+1,1}|_{V_1}]\right) \\
&\ +\frac{1}{2}\sum_{\#V_1=l, \, \#V_2=1}
    \left([K_{l+1,1}|_{V_1}]\otimes[K_{l+1,1}|_{V_2}] + 
          [K_{l+1,1}|_{V_2}]\otimes[K_{l+1,1}|_{V_1}]\right) \\
&\ +\frac{1}{2}\sum_{\text{others}}
    \left([K_{l+1,1}|_{V_1}]\otimes[K_{l+1,1}|_{V_2}] + (-1)^{\#W_1\#W_2} 
          [K_{l+1,1}|_{V_2}]\otimes[K_{l+1,1}|_{V_1}]\right) \\
&=[K_l]\otimes[\circ] + l[K_{l,1}]\otimes[\bl] + O.
\end{split}
\end{align}

Now we prove \eqref{eq:chr:DKn1} by induction on $m$. 
When $m=2$, we have
\begin{align}
&(\Delta \otimes \id)\circ\Delta(K_{n+1,1}) 
 = \Delta(K_n)\otimes[\circ] + n\Delta(K_{n,1})\otimes[\bl] + (\text{others}) \\
&= n[K_{n-1}]\otimes[\bl]\otimes[\circ] + n[K_{n-1}]\otimes[\circ]\otimes[\bl] 
 + n(n-1)[K_{n-1, 1}]\otimes[\bl]\otimes[\bl] + O,
\end{align}
and \eqref{eq:chr:DKn1} actually holds.
Next, we assume \eqref{eq:chr:DKn1} for $2\le m\le n-1$. 
Then for $m+1$ we have
\begin{align}
&   \Delta^{(m+1)}([K_{n+1,1}]) 
 = (\Delta\otimes\id^{\otimes m})\circ\Delta^{(m)}(K_{n+1,1}) \\
&=  n(n-1)\cdots(n-m+1) \cdot \Delta([K_{n-m+1,1}]) \otimes [\bl]^{\otimes m} \\
&\ +n(n-1)\cdots(n-m+2) \cdot \Delta([K_{n-m+1}]) \otimes 
    \sum_{k=1}^m [\bl]^{\otimes (k-1)} \otimes [\circ] \otimes [\bl]^{\otimes (m-k)}
   + O \\
&= n(n-1)\cdots(n-m+1) \cdot \bigl([K_{n-m}]\otimes[\circ]+(n-m)[K_{n-m,1}]\otimes[\bl]\bigr)
   \otimes[\bl]^{\otimes m} \\
&\ +n(n-1)\cdots(n-m+2) \cdot (n-m+1) [K_{n-m}] \otimes [\bl] \otimes
    \sum_{k=1}^m [\bl]^{\otimes (k-1)} \otimes [\circ] \otimes [\bl]^{\otimes (m-k)}
   +O \\
&=  n(n-1)\cdots(n-m)[K_{n-m,1}] \otimes [\bl]^{\otimes(m+1)} \\
&\ +n(n-1)\cdots(n-m+1)[K_{n-m}] \otimes 
    \sum_{k=1}^{m+1} [\bl]^{\otimes (k-1)} \otimes [\circ] \otimes [\bl]^{\otimes (m+1-k)}
   + O,
\end{align}
where we used \eqref{eq:chr:DKl} and \eqref{eq:chr:DKl1} in the third equality.
Hence we have \eqref{eq:chr:DKn1} for $m+1$, and the proof is completed.
\end{proof}

We give a superspace extension of 
the expansion formula of chromatic symmetric functions \cite[Proposition 7.3.17]{GR}.
For a triple $G=(V,W,E)$, a \emph{proper coloring} is a map $f\colon V \to \{1,2,\dotsc\}$
such that $f(v) \ne f(v')$ for any edge $e=\{v,v'\} \in E$.

\begin{prp}\label{prp:chr:PsiG} 
For a triple $G=(V,W,E)$, 
the chromatic symmetric function in superspace $\Psi_{\ch}([G]) \in \sL$ is expressed as 
\begin{align}\label{eq:chr:xf}
 \Psi_{\ch}([G]) = \sum_f \bfx_f, \quad 
 \bfx_f \ceq \prod_{w \in W} \theta_{f(w)} \prod_{v\in V\backslash W}x_{f(v)}
\end{align}
where $f$ runs over all the proper colorings of $G$.
\end{prp}

\begin{proof}
For a triple $G=(V,W,E)$ with $[G] \in \clG^k$, the formula \eqref{eq:CHSA:Psi-sQ} says 
\begin{align}
 \Psi_{\ch}([G]) = \sum_\alpha \zeta_{\ch,\alpha}([G]) M_\alpha,
\end{align}
where $\alpha=(\alpha_1,\dotsc,\alpha_l)$ runs over dotted compositions of $k$.

To interpret the coefficient $\zeta_{\ch,\alpha}([G])$, we note that, 
for $l \ge 2$, the image of $[G] \in \clG$ under 
the iterated comultiplication $\Delta^{(l-1)}\colon \clG \to \clG^{\otimes l}$ is expressed as 
\begin{align}
 \Delta^{(l-1)}([G]) = \frac{1}{2^{l-1}} 
 \sum_{\substack{(V_1,\dotsc,V_{l}): \\ V_1 \sqcup\cdots\sqcup V_l = V}} \sum_{\sigma} \pm
 [\rst{G}{V_{\sigma(1)}}] \otimes [\rst{G}{V_{\sigma(2)}}] 
 \otimes \cdots \otimes [\rst{G}{V_{\sigma(l)}}],
\end{align}
where $\pm$ denotes some sign, 
and $\sigma$ runs over some permutations of $(1,\dotsc,l)$.
Note that the number of such $\sigma$ is $2^{l-1}$.
The supercharacter $\zeta_{\ch}^{\otimes l}$ sends each summand to $1$ or $0$ depending on 
whether each $E_i \ceq \{\{v, v'\}\in E \mid v, v'\in V_{\sigma(i)}\}$ is empty or not, i.e.,
whether the assignment of color $i$ to the vertices in $V_{\sigma(i)}$ gives a proper coloring of $G$.

Hence, the coefficient $\zeta_{\text{ch}, \alpha}([G])$ of $M_\alpha$ in $\Psi_{\ch}([G])$ 
counts the proper colorings $f$ such that $\# f^{-1}(i)=\alpha_i+\eta_i$ 
and $\# f^{-1}(i) \cap W = \eta_i$ for each color $i \in \{1,2,\dotsc\}$,
where $(\eta_1,\dotsc,\eta_l)=\eta(\alpha)$ is given by \eqref{eq:sQ:eta}.
In other words, for a color $i$ such that $\alpha_i=\dot{s}$ ($s\ge0$), 
the counted $f$ satisfies $\# f^{-1}(i)=s+1$ and $\# f^{-1}(i) \cap W =1$,
and for $i$ such that $\alpha_i=r$ ($r\ge1$), 
$f$ satisfies $\# f^{-1}(i)=r$ and $\# f^{-1}(i) \cap W =0$.
Since $M_\alpha=\sum_{f}\bfx_f$ with $f$ running over such proper colorings, 
we have the desired \eqref{eq:chr:xf}.
\end{proof}

At the end of this note, 
we discuss the expansion of chromatic symmetric functions in superspace. 
We begin by recalling one of our examples of a chromatic symmetric function in superspace:
\begin{align}
 \Psi_{\ch}([\begin{tikzpicture}[scale=0.5,baseline=-3pt]
 \draw[-] (4.0,0) node {$\bl$} -- (4.9,0);
 \node at (5.0,0) {$\circ$};
 \draw[-] (5.15,0.05) -- (6,0.3) node {$\bl$} 
 (6, -0.3) node {$\bl$} -- (5.15,-0.05); \end{tikzpicture}])
 = m_{(\dot{0},3)} + 6m_{(\dot{0},1,1,1)} + 3m_{(\dot{0},2,1)}.
\end{align}
By direct calculation, we can expand this polynomial in superspace 
in terms of elementary symmetric functions in superspace:
\begin{align}
 \Psi_{\ch}([\begin{tikzpicture}[scale=0.5,baseline=-3pt]
 \draw[-] (4.0,0) node {$\bl$} -- (4.9,0);
 \node at (5.0,0) {$\circ$};
 \draw[-] (5.15,0.05) -- (6,0.3) node {$\bl$} 
 (6, -0.3) node {$\bl$} -- (5.15,-0.05); \end{tikzpicture}])
 = \wt{e_1}e_1^2 + 6\wt{e_3}+3\wt{e_2}e_1-\wt{e_1}e_2-\wt{e_0}e_1e_2-3\wt{e_0}e_3.
\end{align}
In general, as illustrated by the example above, chromatic symmetric functions in superspace 
cannot be expanded by elementary symmetric functions in superspace with positive coefficients. 
In the non-super case, whether a symmetric function can be expanded in terms of 
elementary symmetric functions is related to the Stanley--Stembridge conjecture \cite{SS}. 
This conjecture was recently proved by Hikita \cite{H} for all unit interval graphs 
by introducing $(q, t)$-chromatic symmetric functions. 
In the super case, we should consider 
the $e$-positivity of chromatic symmetric functions in superspace.
Since $(q,t)$-chromatic symmetric functions provide a new interpretation of 
chromatic symmetric functions in terms of the affine Hecke algebras of type A, 
the $e$-positivity of chromatic symmetric functions in superspace may play a key role 
in developing a super analogue of the representation theory of the affine Hecke algebras.

\begin{Ack}
The authors would like to thank Professor Hiroaki Kanno 
for providing the impetus for this research, and 
Professor Jaeseong Oh for stimulating comments.
They would also like to thank the referees for valuable comments and 
suggestions on improvements.
This work was financially supported by JST SPRING, Grant Number JPMJSP2125. 
M.H.\ and R.Y.\ would like to take this opportunity to thank the 
“THERS Make New Standards Program for the Next Generation Researchers.”
\end{Ack}


\end{document}